\documentclass[a4paper,11pt,twoside,reqno]{amsart}

\usepackage[utf8]{inputenc}
\usepackage[plainpages=false,pdfpagelabels=true]{hyperref}
\usepackage{amssymb,amsthm}
\usepackage[margin=1in]{geometry}
\usepackage{comment}
\usepackage{tikz}
\usepackage{dsfont}
\usepackage{float}
\usepackage{mathtools}
\usepackage{tablefootnote}
\usepackage{graphicx}
\usepackage{booktabs}
\usepackage{multirow}

\newtheorem{theorem}{Theorem}[section]

\newtheorem{lemma}[theorem]{Lemma}
\newtheorem{corollary}[theorem]{Corollary}
\theoremstyle{definition}

\newtheorem{remark}[theorem]{Remark}

\newtheorem*{thma}{Theorem A}
\newtheorem*{thmb}{Theorem B}
\newtheorem*{thmc}{Theorem C}

\allowdisplaybreaks[1]

\renewcommand{\epsilon}{\varepsilon}
\DeclareMathOperator{\tr}{Tr}
\DeclareMathOperator{\Sh}{S}
\newcommand{\dv}{\text{ }dV}
\newcommand{\N}{\mathbb{N}}
\newcommand{\R}{\mathbb{R}}
\newcommand{\CPn}{\mathbb{CP}^{n}}
\newcommand{\CHn}{\mathbb{CH}^{n}}

\numberwithin{equation}{section}

\title{Polyharmonic hypersurfaces into complex space forms}
\author{José Miguel Balado-Alves}
\address{WWU M\"unster, Mathematisches Institut\\
Einsteinstr. 62\\
48149 M\" unster\\
Germany}
\email{jose.balado@uni-muenster.de \\ ORCiD: 0009-0008-5324-2595}

\date{\today}

\begin{document}

\begin{abstract}
We characterize homogeneous hypersurfaces in complex space forms which arise as critical points of a higher order energy functional. As a consequence, we obtain existence and non-existence results for $\mathbb{CP}^n$ and $\mathbb{CH}^n$, respectively. Moreover, we study the stability of biharmonic hypersurfaces and compute the normal index for a large family of solutions.
\end{abstract}

\maketitle

\section{Introduction}

\emph{Harmonic maps} are critical points of the \emph{energy functional}
\begin{equation*}
     E ( \psi ) = \frac{ 1 }{ 2 } \int_{ M } | d \psi |^2 \, \dv_g,
\end{equation*}
where $ \psi : M \to N $ is a smooth map between two Riemannian manifolds $ ( M, g ) $ and $ ( N, h ) $. Equivalently, $ \psi $ is harmonic if and only if it is a solution to the Euler--Lagrange equation associated with this functional, namely:
\begin{equation*}
    \tau ( \psi ) := \tr \nabla d \psi = 0.
\end{equation*}
The section $ \tau ( \psi ) \in \psi^* TN $ is called the \emph{tension field} of $ \psi $. In particular, if $\psi$ is an isometric immersion, it is harmonic if and only if $ \psi ( M ) $ is a minimal submanifold of $ N $. We refer the reader to the work of Eells and Lemaire~\cite{eells1983selected, eells1988another} for background on the theory of harmonic maps.

The study of higher-order functionals has already been proposed in~\cite{eells1983selected} as a generalization to the classic energy. In the last decades, this topic has raised the interest of many mathematicians, leading to intriguing results from both the analytic and the geometric perspective, see for instance~\cite{branding2021structure, branding2023pseudo, branding2020higher, montaldo2022polyharmonic, montaldo2023surfaces}. If $ r = 2 s $ and $ s \geq 1 $, we define the $ r $-energy functional as
\begin{equation*}
    E_{ 2 s } ( \psi ) := \frac{ 1 }{ 2 } \int_{ M } | \Delta^{ s - 1 } \tau ( \psi ) |^2 \, \dv_g.
\end{equation*}
If, on the other hand, $r=2s+1$, then
\begin{equation*}
    E_{ 2 s + 1 } ( \psi ) := \frac{ 1 }{ 2 } \int_{ M } \sum_{ j = 1 }^n | \nabla _{ e_j } \Delta^{ s - 1 } \tau ( \psi ) |^2 \, \dv_g,
\end{equation*}
where $ \{ e_i \}_{ i = 1 }^m $ is a local orthonormal frame on $ M $ and $ \Delta = d^* d $ denotes the rough Laplacian acting on sections of $ \psi^* TN $. A \emph{polyharmonic map of order $ r $} (in short, \emph{$ r $-harmonic}) is a critical point for the $ r $-energy functional. In particular, if $ \psi $ is an isometric immersion, we say that $ M $ is an \emph{$ r $-harmonic submanifold} of $ N $.  The biharmonic case ($r=2$) has shown itself to be of special interest, and we refer to~\cite{chen2014total, fetcu2010biharmonic, ou2020biharmonic} for an introduction to this field.

Note that any harmonic map is automatically polyharmonic of any order. If, on the contrary, a map is a critical point for the $ r $-energy functional but not for the classic energy, we refer to it as a \emph{proper} $ r $-harmonic map. Specifically, an $ r $-harmonic submanifold is said to be \emph{proper} if it is not minimal.

\smallskip
In this manuscript, we consider Hopf hypersurfaces in a complex space form $ ( N, h, J ) $. A hypersurface $ M $ of $ N $ is said to be \emph{Hopf} if $ - J \xi $ is an eigenvector for the shape operator of $ M $, that is to say, $ \Sh ( - J \xi ) = - \alpha J \xi $ where $ \xi $ is a local choice of unit normal vector and $ \alpha \in \mathbb{ R } $. Hopf hypersurfaces with constant principal curvatures in $ \CHn $ or $ \CPn $ have been classified in~\cite{berndt1989real, kimura1986real}: they are all open parts of homogeneous hypersurfaces.

More specifically, the main results we present here are the following.

\begin{thma}\label{thma}
    Let $ M $ be an open part of a homogeneous hypersurface in a complex space form $N^n(c)$, $ n \geq 2 $. Then $ M $ is proper $ r $-harmonic for $ r \geq 2 $ if and only if
    \begin{equation*}
        \tfrac{4}{c}( \tr \Sh^2 )^2 - 2 ( n + 1 ) ( \tr \Sh^2 ) = ( r - 2 ) \tr \Sh ( \tr \Sh + 3 \alpha ).
    \end{equation*}
\end{thma}
All these hypersurfaces in $\CPn$ and $\CHn$ can be seen as tubes of radius $ t \in I $ over a complex submanifold. This interpretation yields the following result, where $ \{ M_t \}_{ t \in I } $ denotes the family of hypersurfaces of tubes of radius $t$ over a fixed complex submanifold.

\begin{thmb}\label{thmb}
    Let $ \{ M_t \}_{ t \in I } $ be a family of hypersurfaces in $\CPn$ as above. There exist two natural numbers $ r' $ and $ r'' $ such that
    \begin{enumerate}
        \item If $ r > r' $, the family $ \{ M_t \}_{ t \in I } $ contains at least two proper $ r $-harmonic hypersurfaces.
        \item If $ r > r'' $, the family $ \{ M_t \}_{ t \in I } $ contains exactly four proper $ r $-harmonic hypersurfaces for some suitable choice of the complex submanifold.
    \end{enumerate}
\end{thmb}

Explicit bounds for $ r' $ and $ r'' $ are given. Moreover, we show that $ \CHn $ does not admit any polyharmonic hypersurface of this type.

\smallskip
For the biharmonic case, we recover the classification of proper biharmonic homogeneous real hypersurfaces in $\CPn$ given in~\cite{sasahara2019classification}, these are tubes over a totally geodesic $\mathbb{CP}^{n-p}$ in $\CPn$ of certain radii $t_-$ or $t_+$, and obtain the following result regarding their stability.

\begin{thmc}\label{thmc}
    Every homogeneous and proper biharmonic hypersurface in $\CPn$ is unstable. Moreover, there exists $C\equiv C(p)>0$ such that if $n-p > C$ the normal index of the biharmonic tube over $\mathbb{CP}^{n-p}$ in $\CPn$ of radius $t_+$ is exactly $1$.
\end{thmc}

The organization of the document is as follows. In Section~\ref{sectionpreliminaries} we introduce some basic notions on polyharmonic maps and Hopf hypersurfaces. In Section~\ref{sectionequation} we prove Theorem\,A. Theorem\,B and explicit upper bounds for $r'$ and $r''$ are given in Section~\ref{sectionpolynomials}. The proof of Theorem\,C is given in Section~\ref{sectionbiharmonic}.

\smallskip
{\bf Acknowledgments.}
The author would like to thank Stefano Montaldo for interesting and helpful discussions. This work was supported by the Deutsche Forschungsgemeinschaft (DFG, German Research Foundation) under Germany’s Excellence Strategy EXC 2044–390685587, Mathematics M\"unster: Dynamics–Geometry–Structure and the CRC 1442 Geometry: Deformations and Rigidity. This article contains results from the author's PhD thesis.


\section{Preliminaries}\label{sectionpreliminaries}

\subsection{Polyharmonic maps}

We introduce here some basic concepts on polyharmonic maps and the Laplace operator used throughout the manuscript. 

The \emph{$ r $-tension field}, $ \tau_r $, is a higher-order analog of the tension field, in the sense that the system of partial differential equations $ \tau_r ( \psi ) = 0 $ characterizes polyharmonic maps of order $ r $. The following equations depict an explicit formula for the $r$-tension field, see~\cite{maeta2012k} for a reference. Here, $ { \psi : M \to N } $ denotes a smooth map between two Riemannian manifolds $ ( M, g ) $ and $ ( N, h ) $. In order to simplify the notation we write $ \tau_r $ instead of $ \tau_r ( \psi ) $.

If $ r = 2 s, s \geq 1 $, then
\begin{equation}\label{2stensionfield}
    \begin{split}
        \tau_{ 2 s }  :=& \Delta^{ 2 s - 1 } \tau  - R^N ( \Delta^{ 2 s - 2 } \tau , d \psi e_j  ) d \psi e_j  \\
        &+ \sum_{ l = 1 }^{ s - 1 } \big[ R^N ( \Delta^{ s - l - 1 } \tau , \nabla_j \Delta^{ s + l - 2 } \tau  ) d \psi e_j  \\
        &\hspace{0.8cm} - R^N ( \nabla_j \Delta^{ s - l - 1 } \tau , \Delta^{ s + l - 2 } \tau  ) d \psi e_j  \big].
    \end{split}
\end{equation}
On the other hand, if $r=2s$,
\begin{equation}\label{2s+1tensionfield}
    \begin{split}
        \tau_{ 2 s + 1 } :=& \Delta^{ 2 s } \tau - R^N ( \Delta^{ 2 s - 1 } \tau, d \psi e_j ) d \psi e_j \\
        &- \sum_{ l = 1 }^{ s - 1 } \big[ R^N ( \nabla_j \Delta^{ s + l - 1 } \tau, \Delta^{ s - l - 1 } \tau ) d \psi e_j \\
        &\hspace{0.8cm} - R^N ( \Delta^{ s + l - 1 } \tau, \nabla_j \Delta^{ s - l - 1 } \tau ) d \psi e_j \big] \\
        &- R^N ( \nabla_j \Delta^{ s - 1 } \tau, \Delta^{ s - 1 } \tau ) d \psi e_j.
    \end{split}
\end{equation}
Recall that the rough Laplacian reads $ \Delta := -\tr ( \nabla \nabla - \nabla_{ \nabla } ) $.

\smallskip
A proper biharmonic hypersurface $ M $ in $ N $ is said to be \emph{normally stable} if the second variation of the bienergy functional is non-negative for any normal variation with compact support. Ou~\cite{ou2022stability} showed that the second variation of the bienergy for a normal variation, $\delta^2 E(f \xi, f \xi)=Q(f)$, is given by
\begin{equation}\label{hessian}\small
    \begin{split}
        Q ( f ) =& \int_M [ f ( \tr \Sh_t^2 - { \rm Ric }^N ( \xi, \xi ) ) - \Delta f ]^2  + | f \nabla \tr \Sh_t - 2 ( { \rm Ric }^N ( \xi ) )^{ \top } + 2 \Sh_t ( \nabla f ) |^2  \dv_g \\
        +& \int_M f^2 \tr \Sh_t [ ( \nabla^{ N }_{ \xi } { \rm Ric }^N ) ( \xi, \xi ) ) - 2 \tr { \rm R }^N ( \xi, \cdot, \nabla^{ N }_{ \xi } ( \cdot ), \xi ) ] -  4 f^2 \tr \Sh_t \tr { \rm R }^N ( \xi, \Sh_t ( \cdot ), \cdot, \xi ) \dv_g
    \end{split}
\end{equation}
where $ \xi $ is a normal vector field along $ M $.

The Laplace operator $ \Delta $ has proved to be useful in the study of the second variation of biharmonic submanifolds, as shown in~\cite{ou2022stability, montaldo2022second}. For convenience we recall some properties of this elliptic operator and its spectrum, we use~\cite{chavel2006riemannian} as a reference.

Let $ M $ be a compact and connected manifold. The set of eigenvalues for
\begin{equation*}
    \Delta f + \mu f = 0,
\end{equation*}
where $ f \in C^2 ( M ) $, consists of a sequence
\begin{equation*}
    0 = \mu_0 < \mu_1 < \ldots < \mu_{\ell} \to +\infty,
\end{equation*}
and each associated eigenspace is finite-dimensional. Eigenspaces belonging to distinct eigenvalues are orthogonal in $ L^2 ( M ) $, and $ L^2 ( M ) $ is the direct sum of all the eigenspaces. Furthermore, each eigenfunction is in $ C^{ \infty } ( M ) $.

In addition, the following result turned out to be useful for our purposes
\begin{theorem}\cite[Theorem~2.1]{ho2008first}\label{firsteigenvalue}
    Suppose that $ M $ is a compact orientable hypersurface embedded in a compact Riemannian manifold $ N $. If the Ricci curvature of $ N $ is bounded below by a positive constant $ k $, then $ 2 \mu_1 > k - \max_M | \tr\Sh| $ where $ \mu_1 $ is the first non-zero eigenvalue of the Laplacian of $ M $.
\end{theorem}

\subsection{Hopf hypersurfaces}

We give in this subsection a brief introduction to the theory of Hopf hypersurfaces. We use~\cite{cecil2015geometry} as a reference.

Let $ ( N, \langle \cdot, \cdot \rangle, J ) $ be a Kähler manifold and $ \sigma $ be a plane in the tangent space $ T_p N $, $ p \in N $. We write
\begin{equation*}
    K_p ( \sigma ) = R ( X, Y, Y, X )
\end{equation*}
for the sectional curvature, where $ \{ X, Y \} $ is an orthonormal basis of $ \sigma $. If $ \sigma $ is invariant by the almost complex structure $ J $, then $ K_p ( \sigma ) $ is called the \emph{holomorphic sectional curvature of $ \sigma $}. In particular, if $ K ( \sigma ) $ is constant for all planes $ \sigma $ in $ T_p N $ invariant by $ J $ and for all points $ p \in N $, then $ N $ is called a \emph{space of constant holomorphic sectional curvature}.

The Riemannian curvature tensor for a space of constant holomorphic sectional curvature $ c \neq 0 $ can be expressed as
\begin{equation}\label{curvaturetensor}
    \tfrac{ 4 }{ c } R ( X, Y, Z ) = \langle Y, Z \rangle X - \langle X, Z \rangle Y + \langle X ,J Z \rangle J Y - \langle Y, J Z \rangle J X + 2 \langle X, J Y \rangle J Z.
\end{equation}
The sectional curvature for any plane $ \sigma $ in $ T_p N $ spanned by two orthonormal vectors $ X, Y $ reads
\begin{equation*}
    K ( \sigma ) = \tfrac{ c }{ 4 } ( 1 + 3 \langle X, J Y \rangle^2 ).
\end{equation*}
The complex projective space endowed with the Fubini--Study metric or the complex hyperbolic space with the Bergman metric are examples of spaces with constant holomorphic sectional curvature. Moreover, any simply connected complete $ 2 n $-dimensional Kähler manifold of constant holomorphic sectional curvature $ c $ is holomorphically isometric to $ \CPn ( c ) $, $ \mathbb{ C }^n $ or $ \CHn ( c ) $ depending if $ c > 0 $, $ c = 0 $ or $ c < 0 $, respectively. We refer the reader to~\cite{kobayashi1996foundations} for an overview of the general aspects of complex geometry.

\smallskip
In what follows we use $ N^n $ to denote either $ \CHn $ or $\CPn$. Let $ M $ be an $ ( 2 n - k ) $-dimensional Riemannian manifold, $ k < 2 n $, and let $ f : M \to N^n $ be a Riemannian immersion. Write $ BM $ for the bundle of unit normal vectors to $ f ( M ) $ in $ N^n $. We define the \emph{tube of radius $ t > 0 $ over $ M $}, denoted by $ M_t $, as the image of the map $ f_t : BM \to N^n $ defined by
\begin{equation*}
    f_t ( p, \xi ) = \exp_p ( t \xi ).
\end{equation*}
Note that given any $ p \in M $, there is always a neighborhood $ U $ of $ p $ in $ M $ such that for all $ t > 0 $ sufficiently small, the restriction of $ f_t $ to $ BU $ is an immersion onto an $ ( 2 n - 1 ) $-dimensional manifold.

Let now $ M \subset N^{ n } $ be a hypersurface and $ \xi $ a local choice of unit normal vector. We say that $ M $ is a \emph{Hopf hypersurface} if 
\begin{equation*}
    W = - J \xi
\end{equation*}
is a principal vector for the shape operator $ \Sh $, that is, $ \Sh W = \alpha W $ where $ { \alpha \in \R } $ is called the \emph{Hopf principal curvature}. We will refer to $ W $ as the \emph{structure vector}. In a Hopf hypersurface $ M $ of $ \CHn $ or $ \CPn $, the Hopf principal curvature remains constant~\cite[Theorem~6.16]{cecil2015geometry}.

In particular, Hopf hypersurfaces for which all their principal curvatures are constant have been classified by Kimura~\cite{berndt1989real} in $ \CHn $ and by Berndt~\cite{kimura1986real} in $ \CPn $. They are all tubes of a certain radius, as shown in Table~\ref{tablechn} for $ \CHn $ and Table~\ref{tablecpn} for $ \CPn $.

The following property turned out to be useful for our purposes.
\begin{lemma}\cite[Lemma~8.1]{cecil2015geometry}\label{lemmahopf}
    Let $ M $ be a Hopf hypersurface with constant principal curvatures in a complex space form $N^n$. For all eigenvalues $ \lambda, \mu $ which are not the Hopf principal curvature, we have
    \begin{equation*}
        \nabla_X Y \perp T_{ \lambda } \text{ if } X \in T_{ \lambda }, Y \in T_{ \mu }, \lambda \neq \mu,
    \end{equation*}
    where $ T_{ \lambda }, T_{ \mu } $ are the corresponding eigendistributions and $ \nabla $ is the Levi--Civita connection on $ M $.
\end{lemma}

\begin{table}[H]
\scalebox{0.9}{
\begin{tabular}{@{}llll@{}}
\toprule
Type  & Focal submanifold & Principal curvatures & Multiplicities \\ \midrule
$ A_0 $ & Horosphere in $ \CHn $ & \begin{tabular}[c]{@{}l@{}}
                                    $ \alpha = 2 $\\ 
                                    $ \lambda_1 = 1 $ 
                                \end{tabular}                   &  \begin{tabular}[c]{@{}l@{}}
                                                                        $ m_{ \alpha } = 1 $\\ 
                                                                        $ m_1 = 2 n - 2 $ 
                                                                    \end{tabular}                 \\ \midrule
$ A_1 $ & Totally geodesic $ \mathbb{ C H }^{ n - 1 }$ &    \begin{tabular}[c]{@{}l@{}}
                                                                $ \alpha = 2 \coth 2 t $\\ 
                                                                $ \lambda_1 = \tanh t $
                                                            \end{tabular}                   &   \begin{tabular}[c]{@{}l@{}}
                                                                                                    $ m_{ \alpha } = 1 $\\ 
                                                                                                    $ m_1 = 2 n - 2 $
                                                                                                \end{tabular}               \\ \midrule
$ A_1 $ & $ \mathbb{ C H }^{ 0 } $ &    \begin{tabular}[c]{@{}l@{}}
                                            $ \alpha = 2 \coth 2 t $\\ 
                                            $ \lambda_1 = \coth t $
                                        \end{tabular}                   &   \begin{tabular}[c]{@{}l@{}}
                                                                                $ m_{ \alpha } = 1 $\\ 
                                                                                $ m_1 = 2 n - 2 $
                                                                            \end{tabular}               \\ \midrule
$ A_2 $ & \begin{tabular}[c]{@{}l@{}} 
            Totally geodesic $ \mathbb{ C H }^k $, \\ 
            $ 1 \leq k \leq n - 2 $ 
        \end{tabular}                                   &   \begin{tabular}[c]{@{}l@{}}
                                                                $ \alpha = 2 \coth 2 t $\\ 
                                                                $ \lambda_1 = \coth t $\\ 
                                                                $ \lambda_2 = \tanh t $
                                                            \end{tabular}                   &   \begin{tabular}[c]{@{}l@{}}
                                                                                                    $ m_{ \alpha } = 1 $\\ 
                                                                                                    $ m_1 = 2 ( n - k - 1 ) $\\ 
                                                                                                    $ m_2 = 2 k $
                                                                                                \end{tabular}                   \\ \midrule
\text{$ B $\tablefootnote{The radius $ t = \tfrac{ 1 }{ 2 } \ln ( 2 + \sqrt{ 3 } ) $ is not allowed.}} & Totally geodesic $ \mathbb{ R H }^n $  \hspace{0.3cm} &    \begin{tabular}[c]{@{}l@{}}
                                                                                                                                                                        $ \alpha = 2 \tanh 2 t $\\ 
                                                                                                                                                                        $ \lambda_1 = \coth t $\\ 
                                                                                                                                                                        $ \lambda_2 = \tanh t $
                                                                                                                                                                        \end{tabular}               &   \begin{tabular}[c]{@{}l@{}}
                                                                                                                                                                                                            $ m_{ \alpha } = 1 $\\ 
                                                                                                                                                                                                            $ m_1 = n - 1 $\\ 
                                                                                                                                                                                                            $ m_2 = n - 1 $
                                                                                                                                                                                                        \end{tabular}               \\ \bottomrule
\end{tabular}
}
\caption{Hopf hypersurfaces with constant principal curvatures in $ \CHn ( - 4 ) $}
\label{tablechn}
\end{table}

\begin{table}[H]
\scalebox{0.9}{
\begin{tabular}{@{}llll@{}}
\toprule
Type & Focal submanifold & Principal curvatures & Multiplicities \\ \midrule
$ A_1 $ & Totally geodesic $ \mathbb{ C P }^{ n - 1 }$ &    \begin{tabular}[c]{@{}l@{}}
                                                                $ \alpha = 2 \cot 2 t $\\ 
                                                                $ \lambda_1 = - \tan t $ 
                                                            \end{tabular}                   &   \begin{tabular}[c]{@{}l@{}}
                                                                                                    $ m_{ \alpha } = 1 $\\ 
                                                                                                    $ m_1 = 2 n - 2 $ 
                                                                                                \end{tabular}               \\ \midrule
$ A_2 $ & \begin{tabular}[c]{@{}l@{}}
            Totally geodesic $ \mathbb{ C P }^k $\\ 
            $ 1 \leq k \leq n - 2 $ 
        \end{tabular}                               &   \begin{tabular}[c]{@{}l@{}}
                                                            $ \alpha = 2 \cot 2 t $\\ 
                                                            $ \lambda_1 = \cot t $\\ 
                                                            $ \lambda_2 = - \tan t $
                                                        \end{tabular}                   &   \begin{tabular}[c]{@{}l@{}}
                                                                                                $ m_{ \alpha } = 1 $\\ 
                                                                                                $ m_1 = 2 ( n - k - 1 ) $\\ 
                                                                                                $ m_2 = 2 k $
                                                                                            \end{tabular}                   \\ \midrule
$ B $ & Totally geodesic $ \mathbb{ R P }^n $ &    \begin{tabular}[c]{@{}l@{}}
                                                        $ \alpha = 2 \tan 2 t $\\ 
                                                        $ \lambda_1 = - \cot t $\\ 
                                                        $ \lambda_2 = \tan t $
                                                    \end{tabular}                   &   \begin{tabular}[c]{@{}l@{}}
                                                                                            $ m_{ \alpha } = 1 $\\ 
                                                                                            $ m_1 = n - 1 $\\ 
                                                                                            $ m_2 = n - 1 $
                                                                                        \end{tabular}               \\ \midrule
$ C $   & \begin{tabular}[c]{@{}l@{}}
            Segre embedding: \\ 
            $ \mathbb{ C P }^1 \times \mathbb{ C P }^k \hookrightarrow \mathbb{ C P }^{ n }, $\\ 
            $ k \geq 2 $, $ n = 2 k + 1 $ 
        \end{tabular}                                                                            &  \begin{tabular}[c]{@{}l@{}}
                                                                                                        $ \alpha = 2 \cot 2 t $\\ 
                                                                                                        $ \lambda_1 = \cot ( t - \tfrac{ \pi }{ 4 } ) $\\ 
                                                                                                        $ \lambda_2 = \cot ( t - \tfrac{ 3 \pi }{ 4 } ) $\\ 
                                                                                                        $ \lambda_3 = \cot ( t - \tfrac{ \pi }{ 2 } ) $\\ 
                                                                                                        $ \lambda_4 = \cot ( t ) $
                                                                                                    \end{tabular}                                           &   \begin{tabular}[c]{@{}l@{}}
                                                                                                                                                                    $ m_{ \alpha } = 1 $\\ 
                                                                                                                                                                    $ m_1 = 2 $\\ 
                                                                                                                                                                    $ m_2 = 2 $\\ 
                                                                                                                                                                    $ m_3 = n - 3 $\\ 
                                                                                                                                                                    $ m_4 = n - 3 $
                                                                                                                                                                \end{tabular}               \\ \midrule
$ D $ &   \begin{tabular}[c]{@{}l@{}}
            Pl\"ucker embedding: \\ 
            $ G_2 ( \mathbb{ C }^5 ) \hookrightarrow \mathbb{ C P }^9 $ 
        \end{tabular}                                                   &   \begin{tabular}[c]{@{}l@{}}
                                                                                $ \alpha = 2 \cot 2 t $\\ 
                                                                                $ \lambda_1 = \cot ( t - \tfrac{ \pi }{ 4 } ) $\\ 
                                                                                $ \lambda_2 = \cot ( t - \tfrac{ 3 \pi }{ 4 } ) $\\ 
                                                                                $ \lambda_3 = \cot ( t - \tfrac{ \pi }{ 2 } ) $\\ 
                                                                                $ \lambda_4 = \cot ( t ) $
                                                                             \end{tabular}                                          &   \begin{tabular}[c]{@{}l@{}}
                                                                                                                                            $ m_{ \alpha } = 1 $\\ 
                                                                                                                                            $ m_1 = 4 $\\ 
                                                                                                                                            $ m_2 = 4 $\\ 
                                                                                                                                            $ m_3 = 4 $\\ 
                                                                                                                                            $ m_4 = 4 $
                                                                                                                                        \end{tabular}               \\ \midrule
$ E $ &   \begin{tabular}[c]{@{}l@{}}
            Half spin embedding: \\ 
            $ SO ( 10 ) / U ( 5 ) \hookrightarrow \mathbb{ C P }^{ 15 } $ 
        \end{tabular}                                                       &   \begin{tabular}[c]{@{}l@{}}
                                                                                    $ \alpha = 2 \cot 2 t $\\ 
                                                                                    $ \lambda_1 = \cot ( t - \tfrac{ \pi }{ 4 } ) $\\ 
                                                                                    $ \lambda_2 = \cot ( t - \tfrac{ 3 \pi }{ 4 } ) $\\ 
                                                                                    $ \lambda_3 = \cot ( t - \tfrac{ \pi }{ 2 } ) $\\ 
                                                                                    $ \lambda_4 = \cot ( t ) $
                                                                                \end{tabular}                                           &   \begin{tabular}[c]{@{}l@{}}
                                                                                                                                                $ m_{ \alpha } = 1 $\\ 
                                                                                                                                                $ m_1 = 6 $\\ 
                                                                                                                                                $ m_2 = 6 $\\ 
                                                                                                                                                $ m_3 = 8 $\\ 
                                                                                                                                                $ m_4 = 8 $
                                                                                                                                            \end{tabular}               \\ \bottomrule
\end{tabular}
}
\caption{Hopf hypersurfaces with constant principal curvatures in $ \CPn ( 4 ) $}
\label{tablecpn}
\end{table}

In what follows we write ``Hopf hypersurface'' instead of ``Hopf hypersurface with constant principal curvatures'' since we will not deal with the general case.

\subsection{Quartic polynomials}

We found it convenient to briefly recall here some general properties of polynomials and their roots.

A general quartic equation over $ \mathbb{ R } $ is any equation of the form
\begin{equation}\label{generalquarticequation}
    a_4 x^4 + a_3 x^3 + a_2 x^2 + a_1 x + a_0 = 0
\end{equation}
where $ a_4, a_3, a_2, a_1, a_0 \in \mathbb{ R } $ and $ a_4 \neq 0 $. If we divide every term by $ a_4 $ and apply the change of variable
\begin{equation*}
    y = x - \frac{ a_3 }{ 4 a_4 },
\end{equation*}
then~\eqref{generalquarticequation} reads
\begin{equation*}
    y^4 + p_2 y^2 + p_1 y + p_0 = 0
\end{equation*}
where
\begin{equation*}\small
    p_2 = \frac{ 8 a_4 a_2 - 3 b_3^2 }{ 8 a_4^2 }, \quad p_1 = \frac{ a_3^3 - 4 a_4 a_3 a_2 + 8 a_4^2 a_1 }{ 8 a_4^3 }, \quad p_0 = \frac{ 16 a_4 a_3^2 a_2 - 64 a_4^2 a_3 a_1 - 3 a_3^4 + 256 a_4^3 a_0 }{ 256 a_4^4 }.
\end{equation*}
In particular, if $ p_1 = 0 $ then~\eqref{generalquarticequation} can be reduced to a biquadratic equation and solved by
\begin{equation*}
    y_{ 1, 2 } = \pm \sqrt{ \frac{ - p_2 + \sqrt{ p_2^2 - 4 p_0 } }{ 2 } }, \quad y_{ 3, 4 } = \pm \sqrt{ \frac{ - p_2 - \sqrt{ p_2^2 - 4 p_0 } }{ 2 } }.
\end{equation*}

We will make use of the so-called \emph{Cauchy bound} for the roots of a polynomial (see, for example~\cite[Chapter~VII]{marden1949geometry}), which states that an upper bound of the absolute value on the roots of~\eqref{generalquarticequation} is given by
\begin{equation*}
    1 + \max \left\{ \left| \tfrac { a_3 }{ a_4 } \right|, \left| \tfrac { a_2 }{ a_4 } \right|, \left| \tfrac{ a_1 }{ a_4 } \right|, \left| \tfrac{ a_0 }{ a_4 } \right| \right\}.
\end{equation*}

\section{Polyharmonic equation}\label{sectionequation}

The goal of this section is to obtain the $ r $-harmonic equation for the isometric immersion of a Hopf hypersurface into $ N^n ( c ) $. We write
\begin{equation*}
    i_t : M_t \hookrightarrow N^n ( c )
\end{equation*}
for the inclusion of the hypersurface $ M_t $ arising as the tube of radius $ t $ over the corresponding complex submanifold. We use $ e_0 $ for the local unit normal vector field to $ M_t $ in $ N^n ( c ) $ such that if we flow in the direction of $ e_0 $, the radius of the tube increases. Every object in $ M_t $ is written with the subscript $ t $, such as the shape operator $ \Sh_t $ or the Hopf principal curvature $ \alpha_t $. 

\begin{theorem}\label{r-tensionfield}
    A Hopf hypersurface $ M_t $ in $ N^n ( c ) $ is $ r $-harmonic, $ r \geq 2 $, if and only if it is minimal or
    \begin{equation}\label{equationc}
        \tfrac{ 4 }{ c } ( \tr \Sh_t^2 )^2 - 2 ( n + 1 ) ( \tr \Sh_t^2 ) - ( r - 2 ) ( \tr \Sh_t )^2 - 3 \alpha_t ( r - 2 ) \tr \Sh_t = 0.
    \end{equation}
\end{theorem}

In order to prove the theorem above, we first state a couple of auxiliary identities. The first lemma below shows that for any $m \in \mathbb{N}$ the tension field $\tau(i_t)$ is an eigenvector for the $m^{\text{th}}$-iterated rough laplacian.
\begin{lemma}\label{lemmaroughlaplacian}
    For every $ m \in \N $, the following identity holds:
    \begin{equation*}
        \Delta^{ m } \tau( i_t ) = ( \tr \Sh_t^2 )^m  \tau( i_t ).
    \end{equation*}
\end{lemma}
\begin{proof}
    We proceed by induction on $ m $. Take an orthonormal basis $ \{ e_i \}_{ i = 1 }^{ 2 n - 1 } $ of $ T_p M_t $ and extend it to a local orthonormal frame in a neighborhood of $ p $. Since $ \Delta^0 \tau \equiv \tau $, by definition of the tension field we get:
    \begin{equation*}\small
        \begin{split}
            \tau ( i_t ) =& \sum_{ i = 1 }^{ 2 n - 1 } \nabla_{ e_i } d i_{ t } e_i - d i_{ t } \nabla_{ e_i } e_i = \sum_{ i = 1 }^{ 2 n - 1 } \nabla_{ e_i } d i_{ t } e_i - \sum_{ j = 1 }^{ 2 n - 1 } \langle \nabla_{ e_i } e_i, e_j \rangle  d i_{ t } e_j \\
            =& \sum_{ i = 1 }^{ 2 n - 1 } \langle \nabla_{ e_i } d i_{ t } e_i, e_0 \rangle e_0 + \sum_{ j = 1 }^{ 2 n - 1 } \left[ \langle \nabla_{ e_i } d i_{ t } e_i, di_t e_j \rangle - \langle \nabla_{ e_i } e_i, e_j \rangle \right] d i_t e_j = \tr \Sh_t e_0,
        \end{split}
    \end{equation*}
    where in the last identity we used the fact that $ i_t $ is an isometric embedding and applied the Gauss formula. Assume that the property holds for an arbitrary $ m $, then
    \begin{equation*}
        \Delta^{ m + 1 } \tau = \Delta ( \beta_m ( t ) e_0 ),
    \end{equation*}
    where, in order to simplify the notation, we write $ \beta_m ( t ) = \tr \Sh_t \, ( \tr \Sh_t^2 )^m $. Note that since the principal curvatures are constant on $ M_t $, then $ \beta_m $ is constant on $ M_t $. By definition of the rough Laplacian, we obtain:
    \begin{equation*}
        \Delta^{ m + 1 } \tau = \sum_{ i = 1 }^{ 2 n - 1 } \underbrace{ \nabla_{ \nabla_{ e_i } e_i } \beta_m e_0 }_{ ( A ) } - \underbrace{ \nabla_{ e_i } \nabla_{ e_i } \beta_m e_0 }_{ ( B ) }.
    \end{equation*}

    Let us start with the term $ ( A ) $. We have
    \begin{equation}\label{partAgeneral}
        \begin{split}
            \sum_{ i = 1 }^{ 2 n - 1 } \nabla_{ \nabla_{ e_i } e_i } \beta_m e_0 &= \beta_m \sum_{ i, j = 1 }^{ 2 n - 1 } \langle \nabla_{ e_i } e_i, e_j \rangle \nabla_{ e_j } e_0 = - \beta_m \sum_{ i, j = 1 }^{ 2 n - 1 } \langle \nabla_{ e_i } e_i, e_j \rangle \Sh_t ( d i_t e_j ),
        \end{split} 
    \end{equation}
    and therefore
    \begin{equation}\label{termA}
        \sum_{ i = 1 }^{ 2 n - 1 } \langle \nabla_{ \nabla_{ e_i } e_i } \beta_m e_0, e_0 \rangle = 0.
    \end{equation}

    For part $ ( B ) $ we get
    \begin{equation}\label{partBgeneral}
        \sum_{ i = 1 }^{ 2 n - 1 } \nabla_{ e_i } \nabla_{ e_i } \beta_m e_0 = \beta_m \sum_{ i = 1 }^{ 2 n - 1 } \nabla_{ e_i } \nabla_{ e_i } e_0 = - \beta_m \sum_{ i = 1 }^{ 2 n - 1 } \nabla_{ e_i } S_{ t } ( d i_t e_i),
    \end{equation}
    thus
    \begin{equation}\label{termB}
        \sum_{ i = 1 }^{ 2 n - 1 } \langle \nabla_{ e_i } \nabla_{ e_i } \beta_m e_0, e_0 \rangle = - \beta_m \sum_{ i = 1 }^{ 2 n - 1 } \langle \nabla_{ e_i } S_{ t } ( d i_t e_i ), e_0 \rangle = - \beta_m \tr \Sh_t^2.
    \end{equation}
    Once we prove that the tangential part vanishes, equations~\eqref{termA} and~\eqref{termB} give the result. Joining~\eqref{partAgeneral} and~\eqref{partBgeneral}, we have that
    \begin{equation}\label{expressiontan}\small
        \left( \Delta^{ m } \tau( i_t ) \right)^{ \text{tan} } = \beta_m \sum_{ i = 1 }^{ 2 n - 1 } \left( \nabla_{ e_i } S_{ t } ( d i_t e_i ) - \sum_{ j = 1 }^{ 2 n - 1 } \langle \nabla_{ e_i } e_i, e_j \rangle \Sh_t ( di_t e_j ) \right)^{ \text{tan} }.
    \end{equation}
    This expression does not depend on the choice of the basis, since the normal part does neither. Take then an orthonormal basis $ \{ E_i \}_{ i = 1 }^{ 2 n - 1 } $ of $ T_p M_t $ such that $ \{ d i_t E_1 = W, d i_t E_2, \ldots, d i_t E_{ 2 n - 1 } \} $ forms an orthonormal basis of principal vectors of $ T_{ i_{ t } ( p ) } i_t ( M_t ) $ with respective principal curvatures $ { \{ \lambda_1 = \alpha, \lambda_2, \ldots, \lambda_{ 2 n - 1 } \} } $. Extend $ \{ E_i \}_{ i = 1 }^{ 2 n - 1 } $ to a local orthonormal frame, then expression~\eqref{expressiontan} reads
    \begin{equation*}\small
        \left( \Delta^{ m } \tau( i_t ) \right)^{ \text{tan} } = \beta_m \sum_{ i = 1 }^{ 2 n - 1 } \left( \lambda_i \nabla_{ E_i } d i_t E_i - \sum_{ j = 1 }^{ 2 n - 1 } \lambda_j \langle \nabla_{ E_i } E_i, E_j \rangle di_t E_j \right)^{ \text{tan} },
    \end{equation*}
    where we used the property that the eigenvalues are constant on $ i_t ( M_t ) $. Equivalently,
    \begin{equation*}\small
        \begin{split}
            \left( \Delta^{ m } \tau( i_t ) \right)^{ \text{tan} } &= \beta_m \sum_{ i, j = 1 }^{ 2 n - 1 } [ \lambda_i \langle \nabla_{ E_i } d i_t E_i, d i_t E_j \rangle - \lambda_j \langle \nabla_{ E_i } E_i, E_j \rangle ] di_t E_j \\
            &= \beta_m \sum_{ i, j = 1 }^{ 2 n - 1 } [ \lambda_i  - \lambda_j] \langle \nabla_{ E_i } d i_t E_i, d i_t E_j \rangle di_t E_j
        \end{split}
    \end{equation*}
    by the Gauss formula. 
    
    It is clear that the terms with $ \lambda_i = \lambda_j $ vanish. Moreover, Lemma~\ref{lemmahopf} ensures that if $ \lambda_i \neq \lambda_j $, and none of them are the Hopf curvature, then $ \langle \nabla_{ E_i } d i_t E_i, d i_t E_j \rangle = 0 $. Hence, only those terms involving the Hopf curvature, $ \alpha $, are left. In this case,
    \begin{equation*}
        \langle \nabla_{ E_1 } W, di_t E_j \rangle = - \langle \nabla_{ E_1 } J e_0, di_t E_j \rangle = - \langle J \nabla_{ E_1 } e_0, di_t E_j \rangle = \alpha \langle J W, di_t E_j \rangle = 0,
    \end{equation*}
    and
    \begin{equation*}
        \langle \nabla_{ E_i } di_t E_i, W \rangle = - \langle d i_t E_i, \nabla_{ E_i } W \rangle = \langle d i_t E_i, \nabla_{ E_i } J e_0 \rangle = - \lambda_i \langle d i_t E_i, J d i_t E_i \rangle = 0.
    \end{equation*}
    This is, the tangential part vanishes, as we wanted to show.
\end{proof}

\begin{lemma}\label{lemma2}
    For any $ \ell, m \in \mathbb{ N } $, the following identities hold:
    \begin{equation*}\small
        \begin{split}
            &\sum_{ j = 1 }^{ 2 n - 1 } R ( \Delta^m \tau, d i_t e_j, d i_t e_j, e_0 ) = \tfrac{ c }{ 4 } 2 ( n + 1 ) \beta_m ( t ), \\[0.5ex]
            &\sum_{ j = 1 }^{ 2 n - 1 } R ( \Delta^m \tau, \nabla_{ e_j } \Delta^{ \ell } \tau, d i_t e_j, e_0 ) = - \tfrac{ c }{ 4 } \beta_m ( t ) \beta_{ \ell } ( t ) [ \tr \Sh_t + 3 \alpha_t ],
        \end{split}
    \end{equation*}
    where $ \beta_{ \ell } ( t ) = \tr \Sh_t \, ( \tr \Sh_t^2 )^{ \ell } $, $ \ell \in \mathbb{ N } $, and $ \{ e_i \}_{ i = 1 }^{ 2 n - 1 } $ is any local orthonormal frame on $ M_t $. 
\end{lemma}
\begin{proof}
    Note that both terms are independent of the choice of the orthonormal basis. Take then $ \{ E_i \}_{ i = 1 }^{ 2 n - 1 } $ an orthonormal basis of $ T_{ p } M_t $ such that $ \{ d i_t E_1 = W, d i_t E_2, \ldots, d i_t E_{ 2 n - 1 } \} $ is an orthonormal basis of principal vectors in $ T_{ i_{ t } ( p ) } i_t ( M_t ) $, and extend it to a local orthonormal frame in $M_t$. Using formula~\eqref{curvaturetensor} for the curvature tensor of a complex space form we obtain
    \begin{equation*}\small
        \tfrac{ 4 }{ c } R ( e_0, W ) W = e_0 + 3 \langle e_0, J W \rangle J W = 4 e_0.
    \end{equation*}
    On the other hand, if $j\neq 2$,
    \begin{equation*}\small
        \tfrac{ 4 }{ c } R ( e_0, d i_t E_j ) d i_t E_j = e_0 + 3 \langle e_0, J d i_t E_j \rangle J d i_t E_j = e_0,
    \end{equation*}
    where in the last identity we used that $ \langle e_0, J d i_t E_j \rangle = \langle W,  d i_t E_j \rangle = 0 $. Then the result follows from a direct computation after noting that $ { \nabla_{ E_j } \Delta^m \tau = - \beta_m \lambda_j d i_t E_j }$.
\end{proof}

\begin{proof}[Proof of Theorem~\ref{r-tensionfield}]
    By equation~\eqref{2stensionfield}, Lemma~\eqref{lemmaroughlaplacian} and Lemma~\eqref{lemma2} we get
    \begin{equation*}
        \tau_{ 2 s } ( i_t ) = ( \tr \Sh_t^2 )^{ 2 s - 3 } \, [ ( \tr \Sh_t^2 )^2 - 2 ( n + 1 ) \tfrac{ c }{ 4 } ( \tr \Sh_t^2 ) - 2 ( s - 1 ) \tfrac{ c }{ 4 } ( \tr \Sh_t )^2 - 6 \tfrac{ c }{ 4 } \alpha_t ( s - 1 ) \tr \Sh_t ] \tau(i_t).
    \end{equation*}
    If the hypersurface if minimal, then $ \tr \Sh_t = 0 $, so $\tau_{ 2 s } ( i_t )$ vanishes. Note that $ \tr \Sh_t^2 = 0 $ would also imply $ \tr \Sh_t = 0 $. In any other case, we obtain the equation stated in the theorem. A similar computation applies for the $ ( 2 s + 1 ) $-tension field~\eqref{2s+1tensionfield}.
\end{proof}

\begin{corollary}
    There are no proper $ r $-harmonic Hopf hypersurfaces in $ \CHn (c) $.
\end{corollary}
\begin{proof}
    By Theorem~\ref{r-tensionfield} and since $ c < 0 $, we have that
    \begin{equation*}
        \tfrac{ 4 }{ c } ( \tr \Sh_t^2 )^2 - 2 ( n + 1 ) ( \tr \Sh_t^2 ) - ( r - 2 ) ( \tr \Sh_t )^2 - 3 \alpha_t ( r - 2 ) \tr \Sh_t < 0
    \end{equation*}
    due to the fact that the principal curvatures of any hypersurface listed in Table~\ref{tablechn} stay positive for $t >0$.
\end{proof}

\section{Existence of proper polyharmonic hypersurfaces in \texorpdfstring{$ \CPn $}{CPn}}\label{sectionpolynomials}

Let $ M_t $ be a Hopf hypersurface in $ \CPn ( c ) $. Note that the eigenvalues of the shape operator scale with a factor of $ \sqrt{ c } $, so equation~\eqref{equationc} is invariant with respect to the choice of the parameter $ c $. That is to say, the existence of $ r $-harmonic hypersurfaces is not affected by the holomorphic sectional curvature. We will then take $ c = 4 $ to be consistent with Table~\ref{tablecpn}, so the equation obtained in the last section reads:
\begin{equation}\label{equation}\small
    ( \tr \Sh_t^2 )^2 - 2 ( n + 1 ) ( \tr \Sh_t^2 ) - ( r - 2 ) ( \tr \Sh_t )^2 - 3 \alpha_t ( r - 2 ) \tr \Sh_t = 0.
\end{equation}

By Theorem~\eqref{r-tensionfield}, a lengthy computation shows that there is a one-to-one correspondence between $ r $-harmonic Hopf hypersurfaces and roots of the polynomial $P ( x ) = a_4 x^4 + a_3 x^3 + a_2 x^2 + a_1 x + a_0$, where the coefficients $ a_4, a_3, a_2, a_1, a_0 $ and the variable $ x $ are given in Table~\ref{polynomials}.

\begin{table}[h]
\scalebox{0.8}{
\begin{tabular}{@{}ccll@{}}
\toprule
Type & Variable & \multicolumn{2}{c}{Coefficients} \\ \midrule
$ A_1 $ & $ x = \sin^2 t $ & \begin{tabular}[c]{@{}l@{}} 
                                            \mbox{\small $ a_4 = 4 ( n^2 + 3 n ) r - 8 ( n - 1 ) $} \vspace{0.3cm} \\  
                                            \mbox{\small$ a_3 = - 2  ( 2 n^2 + 11 n + 3 ) r $} \vspace{0.1cm} \\ 
                                            \hspace{0.75cm} \mbox{\small $ + 4 ( n^2 + 3 n - 4 ) $} \vspace{0.3cm} \\ 
                                            \mbox{\small$ a_2 =  10 ( n + 1 ) r - 2 ( 3 n - 5 ) $} 
                                        \end{tabular}                                                                   &   \begin{tabular}[c]{@{}l@{}}
                                                                                                                                \mbox{\small $ a_1 = - 4 r - 2 (n + 1 ) $ } \vspace{0.3cm} \\ 
                                                                                                                                \mbox{\small $ a_0 = 1 $} 
                                                                                                                            \end{tabular}  \\                                                   \midrule
$ A_2 $ & $ x = \cos^2 t $ & \begin{tabular}[c]{@{}l@{}}
                                            \mbox{\small$ a_4 = 4 ( n^2 + 3 n ) r - 8 ( n - 1 ) $}\vspace{0.3cm} \\ 
                                            \mbox{\small$ a_3 = - 2  ( 2 n^2 + ( 4 k + 11 ) n + 6 k + 3 ) r $} \vspace{0.1cm} \\      
                                            \hspace{0.75cm} \mbox{\small$ + 4 ( n^2 - ( 2 k - 3 ) n - 4 ) $} \vspace{0.3cm} \\ 
                                            \mbox{\small$ a_2 = 2 ( ( 4 k + 5 ) n + 2 k^2 + 11 k + 5 ) r $} \vspace{0.1cm} \\    
                                            \hspace{0.75cm} \mbox{\small$ + 2 ( ( 2 k - 3 ) n + 4 k^2 + 4 k + 5 ) $} 
                                        \end{tabular}                                                                               &   \begin{tabular}[c]{@{}l@{}} 
                                                                                                                                            \mbox{\small$ a_1 = - 2 ( 2 k^2 + 5 k + 2 ) r $} \vspace{0.1cm} \\    \hspace{0.75cm} \mbox{\small$ - 2 ( ( 2 k + 1 ) n + ( 2 k + 1 )^2 ) $} \vspace{0.3cm} \\ \mbox{\small$ a_0 = ( 2 k + 1 )^2 $} 
                                                                                                                                        \end{tabular} \\                                                            \midrule
$ B $ & $ x = \cos^2 2 t $ & \begin{tabular}[c]{@{}l@{}}
                                            \mbox{\small $ a_4 = n ( n + 3 ) r - 2 ( n - 1 ) $}\vspace{0.3cm} \\ 
                                            \mbox{\small$ a_3 = - ( n^2 + 8 n + 3 ) r + 4 n^2 + 2 n - 10 $} \vspace{0.3cm} \\ 
                                            \mbox{\small$ a_2=  ( 5 n + 7 ) r - 2 ( 5 n - 11 ) $} 
                                        \end{tabular}                                                                           &   \begin{tabular}[c]{@{}l@{}}
                                                                                                                                        \mbox{\small$ a_1 = - 4 r + 2 ( n - 7 ) $} \vspace{0.3cm} \\ 
                                                                                                                                        \mbox{\small$ a_0 = 4 $} \end{tabular} \\                       \midrule
$ C $ & \multicolumn{1}{l}{$ x = \cos^2 2 t $} & \begin{tabular}[c]{@{}l@{}}
                                                                \mbox{\small$ a_4 = n ( n + 3 ) r - 2 ( n - 1 ) $}\vspace{0.3cm} \\ 
                                                                \mbox{\small$ a_3 = - ( n^2 + 7 n + 6 ) r + 4 ( n^2 - 3 n - 4 ) $}\vspace{0.3cm} \\ 
                                                                \mbox{\small$ a_2 = 2 ( 2n + 5 ) r - 2 ( 3 n - 41 ) $}
                                                            \end{tabular}                                                                           &   \begin{tabular}[c]{@{}l@{}}
                                                                                                                                                            \mbox{\small$ a_1 = - 4 r + 4 ( n - 17 ) $}\vspace{0.3cm} \\ 
                                                                                                                                                            \mbox{\small$ a_0 = 16 $}\end{tabular} \\               \midrule
$ D $ & \multicolumn{1}{l}{$ x = \cos^2 2 t $} & \begin{tabular}[c]{@{}l@{}}
                                                                \mbox{\small$ a_4 = 27 r - 4 $}\vspace{0.3cm} \\ 
                                                                \mbox{\small$ a_3 = - 48 r + 11 $}\vspace{0.3cm} \\ 
                                                                \mbox{\small$ a_2 = 25 r + 46 $}
                                                                \end{tabular}                                       &   \begin{tabular}[c]{@{}l@{}}
                                                                                                                            \mbox{\small$ a_1 = - 4 r + 44 $}\vspace{0.3cm} \\ \mbox{\small$ a_0 = 16 $}\end{tabular} \\            \midrule
$ E $ & \multicolumn{1}{l}{$ x = \cos^2 2 t $} & \begin{tabular}[c]{@{}l@{}}
                                                                \mbox{\small$ a_4 = 135 r - 14 $}\vspace{0.3cm} \\ 
                                                                \mbox{\small$ a_3 = - 234 r + 100 $}\vspace{0.3cm} \\ 
                                                                \mbox{\small$ a_2 = 117 r + 184 $}
                                                            \end{tabular}                                               &   \begin{tabular}[c]{@{}l@{}}
                                                                                                                            \mbox{\small$ a_1 = - 18 r - 180 $}\vspace{0.3cm} \\
                                                                                                                            \mbox{\small$ a_0 = 72 $}\end{tabular} \\              \bottomrule
\end{tabular}}
\caption{Polynomials characterizing $ r $-harmonicity}
\label{polynomials}
\end{table}

\begin{remark}
    For the cases of type $ C $, $ D $, and $ E $, it is convenient to write everything in terms of $ 2 t $ instead of $ t $. To this end, we found the following trigonometric identities useful:
    \begin{equation*}
        \begin{split}
            \sum_{ k = 0 }^{ 3 } \cot ( t - k \tfrac{ \pi }{ 4 } ) = 4 \cot 4 t, \quad \sum_{ k = 0 }^{ 3 } \cot^2 ( t - k \tfrac{ \pi }{ 4 } ) = 12 + 16 \cot^2 4 t.
        \end{split}
    \end{equation*}
\end{remark}

Note that if $ n = 1 $, and we plug the Hopf principal curvature of a type $ A_1 $ Hopf hypersurface ($ \alpha_t = 2 \cot 2 t $) in equation~\eqref{equation}, one easily sees that the tube around a point of radius $ t $ is a proper $ r $-harmonic curve if and only if
\begin{equation*}
    \sin^2 2 t = r^{ - 1 }.
\end{equation*}
Since $ \mathbb{ C P }^1 $ is just the Riemann sphere, this agrees with~\cite[Theorem~1.1]{montaldo2018new}. From now on, unless otherwise indicated, we will assume $ n \geq 2 $.  

We point out that, in general, there are two values of $ t $ which will be of special interest since they make the left-hand side of~\eqref{equation} independent of $ r $: these are the solutions for
\begin{equation*}
    \tr \Sh_t = 0,
\end{equation*}
which yields minimal hypersurfaces, and
\begin{equation*}
    \tr \Sh_t + 3 \alpha_t = 0.
\end{equation*}

\begin{theorem}\label{tha1}
    Let $ M_t, 0 < t < \frac{ \pi }{ 2 } $, be a family of type $ A_1 $ Hopf hypersurfaces in $ \CPn $, $ n \geq 2 $. Then:
    \begin{enumerate}
        \item \label{itema} The family $ M_t $ contains at least two proper $ r $-harmonic hypersurfaces for $ r \geq 2 $.
        \item \label{itemb} If $ r \geq 2 n + 13 $, the family $ M_t $ contains exactly four proper $ r $-harmonic hypersurfaces.
    \end{enumerate}
\end{theorem}
\begin{proof}
    The solution for $ \tr \Sh_t = 0 $ is given by
    \begin{equation*}
        t_0 = \arctan \left( \tfrac{ 1 }{ 2 n - 1 } \right)^{ \frac{ 1 }{ 2 } }.
    \end{equation*}
    Define then $ x_0 = \sin^2 t_0 = ( 2 n )^{ - 1 } $. After simplifying, we get that
    \begin{equation*}
        2 n^4 P_{ A_{ 1 } } ( x_0 ) = - ( n - 1 ) ( 2 n - 1 )^2 < 0,
    \end{equation*}
    where  $ P_{ A_{ 1 } } ( x ) $ is defined in Table~\ref{polynomials}. Thus, by Theorem~\ref{r-tensionfield}, every zero $ x^* $ of $ P_{ A_{ 1 } } ( x ) $ lying in the interval $ ( 0, 1 ) $ corresponds to a proper $ r $-harmonic tube of radius $ t^* = \arcsin \sqrt{ x^* } $ over a totally geodesic $ \mathbb{ C P }^{ n - 1 } $. Since $ P_{ A_{ 1 } } ( 0 ) = 1, P_{ A_{ 1 } } ( 1 ) = ( 2 n - 1 )^2 $ and $ P_{ A_{ 1 } } ( x_0 ) < 0 $ with $0< x_0 < 1$, we obtain~\ref{itema}. 
    
    For~\ref{itemb}, the equation $ \tr \Sh_t + 3 \alpha_t = 0 $ is solved by
    \begin{equation*}
        t_2 = \arctan \left( \tfrac{ 2 }{ n + 1 } \right)^{ \frac{ 1 }{ 2 } },
    \end{equation*}
    so it is natural to define $ x_2 = \sin^2 t_2 = 2 ( n + 3 )^{ - 1 } $, obtaining
    \begin{equation*}
        ( n + 3 )^4 P_{ A_{ 1 } } ( x_2 ) = - ( 3 n^2 + 2 n + 11 ) ( n + 7 ) ( n - 1 ) < 0.
    \end{equation*}
    The goal now is to define $ x_1 $ in such a way that we have $ 0 < x_0 < x_1 < x_2 < 1 $ with $ P_{ A_{ 1 } } ( 0 ), P_{ A_{ 1 } } ( x_1 ), P_{ A_{ 1 } } ( 1 ) > 0 $ and $ P_{ A_{ 1 } } ( x_0 ), P_{ A_{ 1 } } ( x_2 ) < 0 $. By continuity, this would yield~\ref{itemb}.
    
    Take $ x_1 = x_0 + (n r)^{ - 1 } $, and write
    \begin{equation*}
        Q ( r ) := 2 n^4 r^4 P_{ A_{ 1 } } ( x_1) = b_4 r^4 + b_3 r^3 + b_2 r^2 + b_1 r + b_0
    \end{equation*}
    where
    \[ \begin{matrix*}[l]%
            b_4 &= ( 2 n - 1 ) ( n^2 - 1 ), &b_1 = 8 ( n^3 + 4 n^2 - 5 n + 4 ), \\[1ex]
            b_3 &= - 2 ( 2 n^4 + n^3 - 2 n^2 + 7 n - 4 ),  &b_0 = - 16 ( n - 1 ), \\[1ex]
            b_2 &= - 4 ( 2 n^3 - 7 n^2 + 9 n - 6 ).
    \end{matrix*}\]%
    Since $ Q ( r ) $ is a fourth-degree polynomial with $ b_4 > 0 $, it suffices to find a bound for its largest root. A classic analysis shows that for $ n \geq 4 $:
    \begin{equation*}
        \max \left\{ \left| \tfrac{ b_0 }{ b_4 } \right|, \left| \tfrac{ b_1 }{ b_4 } \right|, \left| \tfrac{ b_2 }{ b_4 } \right|, \left| \tfrac{ b_3 }{ b_4 } \right| \right\} = \left| \tfrac{ b_0 }{ b_4 } \right| < 2 ( n + 3 ).
    \end{equation*}
    For $ n = 2 $ and $ n = 3 $, $ \left| b_1 b_4^{ - 1 } \right| $ dominates, but since this fraction is bounded by $ 16 $ for every $ n \geq 2 $, we can modify the former bound so that it also covers these cases. For instance, 
    \begin{equation*}
        \max \left\{ \left| \tfrac{ b_0 }{ b_4 } \right|, \left| \tfrac{ b_1 }{ b_4 } \right|, \left| \tfrac{ b_2 }{ b_4 } \right|, \left| \tfrac{ b_3 }{ b_4 } \right| \right\} < 2 ( n + 6 ).
    \end{equation*}
    Hence, using the Cauchy bound, we obtain that condition $ r \geq 2 n + 13 $ suffices to ensure $ P_{ A_{ 1 } } ( x_1 ) < 0 $, obtaining the result.
\end{proof}

\begin{theorem}
    Let $ M_t, 0 < t < \frac{ \pi }{ 4 } $, be a family of type $ B $ Hopf hypersurfaces in $ \CPn $, $ n \geq 2 $. Then:
    \begin{enumerate}
        \item If $ r \geq \min \{ 6001, 12 n^2 + 16 n - 19 \} $, the family $ M_t $ contains at least two proper $ r $-harmonic hypersurfaces.
        \item If $ r \geq \max \{ 6001, 12 n^2 + 16 n - 19 \} $, the family $ M_t $ contains exactly four proper $ r $-harmonic hypersurfaces.
    \end{enumerate}
\end{theorem}
\begin{proof}
    In this case, we have that $ P_{ B } ( 0 ) = 4 $ and $ P_{ B } ( 1 ) = 4 ( n - 1 )^2 $. The equation $ \tr \Sh_t = 0 $ is solved by
    \begin{equation*}
        t_1 = \frac{ 1 }{ 2 } \arctan( \sqrt{ n - 1 } ),
    \end{equation*}
    and if we plug $ x_1 = \cos^2 2 t_1 = n^{ - 1 } $ in $ P_B $ we see that 
    \begin{equation*}
        n^4 P_B ( x_1 ) = 2 ( 3 n - 1 ) ( n - 1 )^3 > 0.
    \end{equation*}
    Thus, this is not enough to ensure the existence of proper $ r $-harmonic hypersurfaces. To solve this problem, define
    \begin{equation*}
        x_0 = 2 r^{ - 1 } \quad \text{and} \quad x_2 = 1 - 5 r^{ - 1 }.
    \end{equation*}
    Then we have $ P_B ( 0 ), P_B ( x_1 ), P_B ( 1 ) > 0 $ and
    \begin{equation*}
        \begin{split}
            P_B ( x_0 ) &= - 4 + O ( r^{ - 1 } ), \\
            P_B ( x_2 ) &= - ( n - 1 )^2 + O ( r^{ - 1 } )
        \end{split}
    \end{equation*}
    as $ r \to \infty $. Hence, a similar argument as in the proof of the previous theorem using the Cauchy bound shows that, to ensure $ P_B ( x_0 ) < 0 $, it is enough to ask for
    \begin{equation*}
        r \geq 12 n^2 + 16 n - 19.
    \end{equation*}
    On the other hand, for $ P_B ( x_0 ) < 0 $ to hold it suffices to take 
    \begin{equation*}
        r \geq 126 + \frac{ 125 ( 23 n + 1 ) }{ ( n - 1 )^2 }.
    \end{equation*}
    Since the right hand side of the inequality above is bounded by $ 6001 $, the result follows.
\end{proof}

\begin{theorem}
    Let $ M_t, 0 < t < \frac{ \pi }{ 4 } $, be a family of type $ C $ Hopf hypersurfaces in $ \CPn $, $ n \geq 5 $. Then:
    \begin{enumerate}
        \item If $ r \geq 300 $, the family $ M_t $ contains at least two proper $ r $-harmonic hypersurfaces.
        \item If $ 4 r \geq 1125 n^2 + 375 n - 1996 $, the family $ M_t $ contains exactly four proper $ r $-harmonic hypersurfaces.
    \end{enumerate}
\end{theorem}
\begin{proof}
    Note that $ P_{ C } ( 0 ) = 16 $ and $ P_C ( 1 ) = 4 ( n - 2 )^2 $. Define
    \begin{equation*}
        x_0 = 5 r^{ - 1 }, \quad x_1 = 2 n^{-1} \quad \text{and} \quad x_2 = 1 - 4 r^{ - 1 }.
    \end{equation*}
    Since
    \begin{equation*}
        \begin{split}
            P_C ( x_0 ) &= - 4 + O ( r^{ - 1 } ), \\[1ex]
            P_C ( x_2 ) &= - 12 ( n - 2 ) + O ( r^{ - 1 } )
        \end{split}
    \end{equation*}
     as $ r \to \infty $, and
     \begin{equation*}
          n^4 P_C ( x_1 ) = 8 ( 3 n - 1 ) ( n - 1 ) ( n - 2 )^2,
     \end{equation*}
     the result comes after using the Cauchy bound as we did before.
\end{proof}

\pagebreak
\begin{theorem}
    Let $ M_t, 0 < t < \frac{ \pi }{ 4 } $, be a family of type $ D $ Hopf hypersurfaces in $ \mathbb{ C P }^9 $. Then:
    \begin{enumerate}
        \item If $ r \geq 32 $, the family $ M_t $ contains at least two proper $ r $-harmonic hypersurfaces.
        \item If $ r \geq 89 $, the family $ M_t $ contains exactly four proper $ r $-harmonic hypersurfaces.
    \end{enumerate}
\end{theorem}
\begin{proof}
    Note that $ P_{ D } ( 0 ) = 16 $ and $ P_D ( 1 ) = 25 $. Define
    \begin{equation*}
        x_0 = 5 r^{ - 1 }, \quad x_1 = \tfrac{ 4 }{ 9 } \quad \text{and} \quad x_2 = 1 - 3 r^{ - 1 }.
    \end{equation*}

    Since $ P_D ( x_1 ) > 0 $, the only conditions we need to ensure are $ P_D ( x_0 ) < 0 $ and $ P_D ( x_2 ) < 0 $. A simple computation shows
    \begin{equation*}
        L ( r ) := \frac{ d^2 }{ d r^2 } r^4 P_D ( Z_0 ) = - 48 r^2 + 2430 r - 9700.
    \end{equation*}
    Hence, we know that if $ r \geq 47 $ then the function $ L ( r ) $ is concave on $ r $. As $ { P_D ( 88 ) > 0 > P_D ( 89 ) } $, we obtain that $ P_D ( x_0 ) < 0 $ if $ r \geq 89 $. A similar argument applies to the other condition.     
\end{proof}

\begin{theorem}
    Let $ M_t, 0 < t < \frac{ \pi }{ 4 } $, be a family of type $ E $ Hopf hypersurfaces in $ \mathbb{ C P }^{ 15 } $. Then:
    \begin{enumerate}
        \item If $ r \geq 27 $, the family $ M_t $ contains at least two proper $ r $-harmonic hypersurfaces.
        \item If $ r \geq 100 $, the family $ M_t $ contains exactly four proper $ r $-harmonic hypersurfaces.
    \end{enumerate}
\end{theorem}
\begin{proof}
    The proof follows in the same way as in the previous theorem after defining
    \begin{equation*}
        x_0 = 5 r^{ - 1 }, \quad x_1 = \tfrac{ 2 }{ 5 } \quad \text{and} \quad x_2 = 1 - 4 r^{ - 1 }.
    \end{equation*}
\end{proof}

Only the case of type $A_2$ hypersurfaces is left. Even though this is the richest case, the dependence on one additional parameter $k$ leads to some extra difficulties. In~\cite{balado2024explicit}, the author constructed two uncountable families of explicit harmonic selfmaps by translation of these hypersurfaces in the normal direction.

\begin{theorem}\label{tha2}
    Let $ M_t, 0 < t < \frac{ \pi }{ 2 } $, be a family of type $ A_2 $ Hopf hypersurfaces in $ \CPn $, $ n \geq 2 $, $ 1 \leq k \leq n - 2 $. Then:
    \begin{enumerate}
        \item\label{A21} The family $ M_t $ contains at least two proper $ r $-harmonic hypersurfaces for $ r \geq 2 $. If, in particular, $ 2 k = n - 1 $, these hypersurfaces are tubes of radius $ t $ over a totally geodesic $ \mathbb{ C P }^{ \frac{ n - 1 }{ 2 } } $, with $ t $ given by
            \begin{equation*}\small
                \cos 4 t = \frac{ n \sqrt{ \omega } - 2 ( 2 n - 1 ) ( n + 1 ) }{ n ( n + 3 ) r - 2 ( n - 1 ) } 
            \end{equation*}
            where
            \begin{equation*}
                \omega = ( n + 3 )^2 r^2 - 8 n ( n + 3 ) r + 16 ( n^2 + 2 n - 2 ).
            \end{equation*}
        \item\label{A22} Let
            \begin{equation*}
                k_1 := \frac{ 5 n^2 - 4 n + 2 - n \sqrt{ 13 n^2 - 8 n + 4 } }{ 4 ( n - 1 ) }.
            \end{equation*}
            If $ k < k_1 $ and
            \begin{equation*}
                r \geq 4 ( 22 k^4 + 85 k^3 + 123 k^2 + 54 k + 8 ) k^2,
            \end{equation*}
            the family $ M_t $ contains exactly four proper $ r $-harmonic hypersurfaces. 
        \item\label{A23} Let
            \begin{equation*}
                k_2 := \frac{ n \sqrt{ 13 n^2 - 8 n + 4 } - n^2 - 4 n + 2 }{ 4 ( n - 1 ) }.
            \end{equation*}
            If $ k > k_2 $ and
            \begin{equation*}
                r \geq 4 ( 6 k^4 + 19 k^3 + 39 k^2 + 8 k + 2 ) (2k+1) k,
            \end{equation*}
            the family $ M_t $ contains exactly four proper $ r $-harmonic hypersurfaces.
    \end{enumerate}
\end{theorem}
\begin{proof}
    If $ 2 k = n - 1 $, equation $ P_{ A_{ 2 } } ( x ) = 0 $ reads
    \begin{equation}\label{explicitA2}
        \begin{split}
            & 4 [ n ( n + 3 ) r - 2 ( n - 1 ) ] \, x^{ 4 } - 8 [ n ( n + 3 ) r - 2 ( n - 1 ) ] \, x^{ 3 } \\[1ex] 
            &+ [ 5 n ( n + 3 ) r + 4 ( n^2 - 2 n + 2 ) ] \, x^{ 2 } - [ n ( n + 3 ) r + 4 n^2 ] \, x + n^2 = 0.
        \end{split}
    \end{equation}
    Note that, in this case, the relation $ a_3^3 - 4 a_4 a_3 a_2 + 8 a_4^2 a_1 = 0 $ holds for the coefficients of $ P_{ A_{ 2 } } ( x ) $. Thus~\eqref{explicitA2} can be reduced to a biquadratic equation, obtaining that the real zeros are given by
    \begin{equation*}
        x_{ \pm }= \tfrac{ 1 }{ 2 } \pm \tfrac{ 1 }{ 2 } [ 2 n ( n + 3 ) r - 4 ( n - 1 ) ]^{ - \frac{ 1 }{ 2 } } [ n ( n + 3 ) r - 4 (n^2 + n - 1 ) + n \sqrt{ \omega } ]^{ \frac{ 1 }{ 2 } }.
    \end{equation*}

    Fix now $ k \neq \tfrac{ n - 1 }{ 2 } $. Note that $ P_{ A_{ 2 } } ( 0 ) = ( 2 k + 1 )^2 $ and $ P_{ A_{ 2 } } ( 1 ) = ( 2 k - 2 n  + 1 )^2 $, so if we find $ x_1 \in ( 0, 1 ) $ such that $ P_{ A_{ 2 } } ( x_1 ) < 0 $, we will have ensured the existence of at least two $ r $-harmonic hypersurfaces. The equation $ \tr \Sh_t = 0 $ is solved by
    \begin{equation*}
        t^* = \arctan \left( \tfrac{ 2 n - 2 k - 1 }{ 2 k + 1 } \right)^{ \frac{ 1 }{ 2 } }.
    \end{equation*}
    Plugging $ x^* = \cos^2 t^* = ( 2 k + 1 ) ( 2 n )^{ - 1 } $ in $ P_{ A_{ 2 } } $ we obtain
    \begin{equation*}
        2 n^4 P_{ A_{ 2 } } ( x^* ) = - ( n - 1 ) ( 2 n - 2 k - 1 )^2 ( 2 k + 1 )^2 < 0,
    \end{equation*}
    from where~\ref{A21} follows.

    Assume now that $ k < k_1 $, where $ k_1 $ is defined as stated in the theorem, and define
    \begin{equation*}
        x_0 = x^*, \quad x_1 = x^* + r^{ - 1 } \quad \text{and} \quad x_2 = 1 - r^{ - 1 }.
    \end{equation*}
    A direct computation shows that
    \begin{equation*}
        \begin{split}
            2 n^4 P_{ A_{ 2 } } ( x_1 ) &= ( 2 k + 1 ) ( 2 n - 2 k - 1 ) \eta_1 ( n, k ) + O ( r^{ - 1 } ) \\[1ex]
            P_{ A_{ 2 } } ( x_2 ) &= - 3 ( 2 n - 2 k - 1 ) + O ( r^{ - 1 } )
        \end{split}
    \end{equation*}
    as $ n, r \to \infty $, where
    \begin{equation*}
        \eta_1 ( n, k ) = 4 ( n - 1 ) k^2 - ( 10 n^2 - 8 n + 4 ) k + 3 n^3 - 5 n^2 + 3n - 1,
    \end{equation*}
    which satisfies $ \eta_1 ( n, k ) > 0 $ since $ k < k_1 $.
    
    On the other hand, if $ k > k_2 $, take
    \begin{equation*}
        x_0 = r^{ - 1 }, \quad x_1 = x^* - r^{ - 1 } \quad \text{and} \quad x_2 = x^*.
    \end{equation*}
    Hence
    \begin{equation*}
        \begin{split}
            P_{ A_{ 2 } } ( x_0 ) &= - 3 ( 2 k + 1 ) + O ( r^{ - 1 } ) \\[1ex]
            2 n^4 P_{ A_{ 2 } } ( x_1 ) &= \eta_2(n,k) ( 2 n - 2 k + 1 ) ( 2 k + 1 ) + O ( r^{ - 1 } )
        \end{split}
    \end{equation*}
    as $ n, r \to \infty $, where
    \begin{equation*}
        \eta_2 ( n, k ) = 4 ( n - 1 ) k^2 + 2 ( n^2 + 4 n - 2 ) k - 3 n^3 + n^2 + 3 n - 1,
    \end{equation*}
    which also satisfies $ \eta_2 ( n, k ) > 0 $ since $ k > k_2 $. 

    In this case the Cauchy bound leads to unmanageable expressions. We avoid this by using a more direct method, paying the price of a worse bound. Since the strategy is similar for each condition, we only write here the argument to ensure condition $ P_{ A_{ 2 } } ( 1 - r^{ - 1 } ) < 0 $ when $ k < k_1 $. We have that
    \begin{equation*}
        \begin{split}
            Q ( r ) := r^4 P_{ A_{ 2 } } ( 1 - r^{ - 1 } ) =& - 3 ( 2 n - 2 k - 1 ) r^4 - 2 ( 2 k^2 - ( 2 k + 13 ) n + 11 k + 5 ) r^3 \\
            &+ 4 ( 2 k^2 - ( 3 k + 11 ) n + 5 k + 4 ) r^2 + 8 ( ( k + 4 ) n - 2 ) r - 8 ( n - 1 ).
        \end{split}
    \end{equation*}
    Since, in any case, $ 3 ( 2 n - 2 k - 1 ) > 1 $, a very conservative bound can be given by
    \begin{equation*}
        Q ( r ) < - r^4 + ( 2 k + 13 ) n r^3 + 4 ( 2 k^2 + 5 k + 4 ) r^2 + 8 n ( k + 4 ) r < - r^4 + ( 18 n^2 + 65 n + 16 ) r^3
    \end{equation*}
    since $ r \geq 2 $ and $ k < n $. Hence, it suffices to take
    \begin{equation*}
        r > 18 n^2 + 65 n + 16.
    \end{equation*}
    A similar computation holds for the rest of the conditions. 
\end{proof}

\section{Biharmonic Hopf hypersurfaces}\label{sectionbiharmonic}

In this section we focus on the biharmonic case. Note that equation~\eqref{equation} for the case $r=2$ reduces to
\begin{equation}\label{biharmonicequation}
    \tr \Sh_t^2 = 2(n+1).
\end{equation}
With this equation we recover the classification of biharmonic Hopf hypersurfaces in $\CPn$, first studied in~\cite{ichiyama2009} and corrected later in~\cite{sasahara2019classification}.

\begin{theorem}\cite[Theorem~2.4]{sasahara2019classification}\label{classificationbiharmonic}
    A Hopf hypersurface in $ \CPn $ is proper biharmonic if and only if it is an open part of a tube over a totally geodesic $ \mathbb{ C P }^{ n-p } $ in $ \CPn $, $ 1 \leq p \leq n - 1 $, with radius $ t $ given by
        \begin{equation*}
            \cos^2 t_{ \pm } = \frac{ 3 ( n + 1 ) - 2 p \pm \sqrt{ n^2 + 6 n - 4 ( n + 1 ) p + 4 p^2 + 5 } }{ 4 ( n + 1 ) }.
        \end{equation*}
\end{theorem}

Regarding stability, we obtain the following result.

\begin{theorem}\label{stability}
    Every homogeneous and proper biharmonic hypersurface in $\CPn$ is unstable. Moreover, there exists $C\equiv C(p)>0$ such that if $n-p > C$ the normal index of the biharmonic tube over $\mathbb{CP}^{n-p}$ in $\CPn$ of radius $t_+$ is exactly $1$.
\end{theorem}
\begin{proof}
    Write $M_t$ for the biharmonic tube over $\mathbb{CP}^{n-p}$ in $\CPn$. Since $ \CPn $ endowed with the Fubini--Study metric is an Einstein manifold with $ { \rm Ric } ( \cdot, \cdot ) = 2 ( n + 1 ) g ( \cdot, \cdot ) $, we obtain the following relations:
    \begin{equation*}
            { \rm Ric } ( \xi, \xi ) = 2 ( n + 1 ), \quad
            ( { \rm Ric } ( \xi ) )^{ \top } = 0, \quad
            ( \nabla_{ \xi } { \rm Ric } ) ( \xi, \xi ) =0.
    \end{equation*}
    Moreover, for any local orthonormal frame $ \{ e_i \}_{ i = 0 }^{ 2 n - 1 } $ on $ \CPn $ we have
    \begin{equation*}
        \tr { \rm R } ( \xi, \cdot, \nabla_{ \xi } ( \cdot ), \xi ) = \sum_{ i, j = 0 }^{ 2 n - 1 } \langle \nabla_{ \xi } e_i, e_j \rangle { \rm R } ( \xi, e_i, e_j, \xi )=\sum_{ i = 0 }^{ 2 n - 1 } \langle \nabla_{ \xi } e_i, e_i \rangle { \rm R } ( \xi, e_i, e_i, \xi ) = 0.
    \end{equation*}
    Then, since
    \begin{equation*}
        \tr { \rm R } ( \xi, \Sh_t ( \cdot ), \cdot, \xi ) = \tr \Sh_t + 3 \alpha_t
    \end{equation*}
    and $ \tr \Sh_t $ is constant on $ M_t $, equations~\eqref{biharmonicequation} and~\eqref{hessian} yield
    \begin{equation*}
        Q ( f ) = \int_M [ ( \Delta f )^2 + 4 | \Sh_t ( \nabla f ) |^2 - 4 f^2 \tr \Sh_t ( \tr \Sh_t + 3 \alpha_t ) ] \dv_g.
    \end{equation*}
    By the Sturm-Liouville's decomposition we have that
    \begin{equation*}
        C^{ \infty } ( M_t ) = \oplus_{ i = 0 }^{ \infty } E_{ \mu_i }
    \end{equation*}
    in $L^2$, where $ E_{ \mu_i } $ denotes the eigenspace of the Laplacian on $ M_t $ with respect to the eigenvalue $ \mu_i $. For a detailed discussion on the orthogonal decomposition of eigenspaces of the Laplacian on compact manifolds see~\cite[Theorem~III.9.1]{chavel2006riemannian}.
    
    We see that for any constant function $f_0$ we have $Q(f_0)<0$. Since $\mu_0=0$ has multiplicity $1$, the index is at least $1$ and therefore $M_t$ is unstable. For any the other eigenfunction $f$ note that
    \begin{equation*}
        \int_M 4 | \Sh_t ( \nabla f ) |^2  \dv_g \geq 4 \lambda_{ \text{min} }^2 \int_M | \nabla f |^2  \dv_g = - 4 \lambda_{ \text{min} }^2 \int_M f \Delta f  \dv_g,
    \end{equation*}
    where $ \lambda_{ \text{min} }^2 $ represents the minimum between all squared principal curvatures. With this
    \begin{equation*}
        Q(f)\geq \int_M [(\Delta f)^2 - 4 \lambda_{\text{min}}^2 f \Delta f -4f^2 \tr \Sh_t(\tr \Sh_t+3\alpha_t)] \dv_g.
    \end{equation*}

     Since in our case we have limited knowledge about the spectrum of the Laplacian on $M_t$, we make use of the following bound for the first eigenvalue:
    \begin{equation*}
        \mu_1 \geq ( n + 1 ) - \tfrac{ 1 }{ 2 } | \tr \Sh_t |.
    \end{equation*}
    This is just an application for our particular case of Theorem~\ref{firsteigenvalue}. Hence, if we denote by $ f_1 $ the eigenfunction corresponding to the eigenvalue $ \mu_1 $, we have
    \begin{equation*}
        \begin{split}
            Q ( f_1 ) &\geq \int_M [ ( \mu_1 )^2 + 4 \lambda_{ \text{min} }^2 \mu_1 - 4 \tr \Sh_t ( \tr \Sh_t + 3 \alpha_t ) ] f^2 \dv_g  \\
            &\geq \int_M [ ( ( n + 1 ) - \tfrac{ 1 }{ 2 } | \tr \Sh_t | )^2 + 4 \lambda_{ \text{min} }^2 ( ( n + 1 ) - \tfrac{ 1 }{ 2 } | \tr \Sh_t | ) - 4 \tr \Sh_t ( \tr \Sh_t + 3 \alpha_t ) ] f^2 \dv_g.
        \end{split}
    \end{equation*}
    After rearranging, we see that for $ Q ( f_1 ) > 0 $ to hold, it is enough to have
    \begin{equation}\label{conditionstability}
        ( n + 1 ) ( 4 \lambda_{ \text{min} }^2 + n + 1 ) > \tfrac{ 15 }{ 4 } ( \tr \Sh_t )^2 + ( 2 \lambda_{ \text{min} }^2 + n + 1 ) | \tr \Sh_t | + 12 \alpha_t \tr \Sh_t.
    \end{equation}
    Now, if $ \mu_{ \ell } $ denotes the $ \ell^{ \text{th} } $-eigenvalue, and since $ \mu_{ \ell } > \mu_1 $, it follows that
    \begin{equation*}
        ( \mu_{ \ell } )^2 + 4 \lambda_{ \text{min} }^2 \mu_{ \ell } - 4 \tr \Sh_t ( \tr \Sh_t + 3 \alpha_t ) > ( \mu_1 )^2 + 4 \lambda_{ \text{min} }^2 \mu_1 - 4 \tr \Sh_t ( \tr \Sh_t + 3 \alpha_t ),
    \end{equation*}
    so to check if $ Q ( f_{ \ell } ) > 0 $ for any $ \ell \geq 1 $, it is enough to ensure~\eqref{conditionstability}.

    Note now that for the values $t_+$ given in~\ref{classificationbiharmonic} the following asymptotic relations hold:
    \begin{align*}
        &4 \cot^2 2t_+ = \tfrac{ 2 n }{ 2 p - 1 } + o ( n ) \,
        &&\cot^2 t_+ = \tfrac{ 2 n }{ 2 p - 1 } + o ( n ) \\
        &\tan^2 t_+ = \tfrac{ 2 p - 1 }{ 2 n } + o ( n^{ - 1 } ) \,
        &&\tr \Sh_{ t_+ } = \tfrac{ 2 \sqrt{ 4 p - 2 } }{ \sqrt{ n } } + o ( n^{ - \frac{ 1 }{ 2 } } )
    \end{align*}
    as $ n \to \infty $. Therefore, there exists $ n_1 > p $ such that $ \lambda_\text{min}^2 = \tan t_+ $ for $ n > n_1 $.
    Using this, we get
    \begin{equation*}
        ( n + 1 ) ( 4 \lambda_{ \text{min} }^2 + n + 1 ) = n^2 + o ( n^2 )
    \end{equation*}
    and
    \begin{equation*}
        \tfrac{ 15 }{ 4 } ( \tr \Sh_t )^2 + ( 2 \lambda_{ \text{min} }^2 + n + 1 ) | \tr \Sh_t | +12 \alpha_t \tr \Sh_t = 48+2\sqrt{4p-2}\sqrt{n} + o(\sqrt{n})
    \end{equation*}
    as $ n \to \infty $, obtaining the result.
\end{proof}

\bibliographystyle{plain}
\bibliography{Bibliography.bib}

\end{document}